\documentclass[10pt]{amsart}

\usepackage[margin=1.5in]{geometry}                
\usepackage{graphicx}
\usepackage{amssymb}
\usepackage{epstopdf}
\usepackage{color}
\usepackage{xcolor}
\usepackage{ulem}
\usepackage{enumerate}
\usepackage{mathtools}
\usepackage{enumitem}
\usepackage{amsmath}
\usepackage{braket}
\usepackage{mathrsfs}  
\usepackage{hyperref}

\usepackage [english]{babel}
\usepackage [autostyle, english = american]{csquotes}
\MakeOuterQuote{"}

\newcommand{\R}{\mathbb{R}}
\newcommand{\N}{\mathbb{N}}
\newcommand{\Z}{\mathbb{Z}}

\newcommand{\Q}{\mathbb{Q}}

\newtheorem{theorem}{Theorem}[section]
\newtheorem{corollary}[theorem]{Corollary}

\newtheorem{fact}[theorem]{Theorem}

\newtheorem{lemma}[theorem]{Lemma}

\newtheorem{obs}[theorem]{Observation}

\newtheorem{defn}[theorem]{Definition}

\newtheorem{remark}[theorem]{Remark}
\newtheorem*{Theorem1}{Theorem \ref{pressure is computable from above}}
\newtheorem*{Theorem2}{Theorem \ref{pressure is computable}}
\newtheorem*{Theorem3}{Theorem \ref{compute time}}
\newtheorem*{Theorem4}{Theorem \ref{compute time2}}
\newtheorem*{Theorem5}{Theorem \ref{ground state energy}}
\newtheorem*{Cor1}{Corollary \ref{computable from above cor ground state}}
\newtheorem*{Cor2}{Corollary \ref{ground state entropy}}

\title{Computability of pressure for subshifts on countable amenable groups}

\date{19 July 2024}
\begin{document}

\maketitle
\begin{center}
\author{C. Evans Hedges and Ronnie Pavlov}
\end{center}

\begin{abstract} 
There are a variety of results in the literature proving forms of computability for topological entropy and pressure on subshifts. In this work, we prove two quite general results, showing that topological pressure is always computable from above given an enumeration for a forbidden list inducing the subshift, and that for strongly irreducible shifts of finite type, topological pressure is computable. Our results apply to subshifts on all finitely generated amenable groups with decidable word problem and generalize several previous results which applied only to $\mathbb{Z}^d$-subshifts.


As corollaries, we obtain some results related to ground state energy and entropy, proving that the map sending $\phi$ to 
$\sup_{\mu \in M_\sigma(X)} \int \phi d\mu$ is computable/computable from above when $P_X(\phi)$ is, and that the map sending $\phi$ to its ground state/residual entropy is computable from above when $P_X(\phi)$ is computable.

We conclude by giving explicit bounds on computation time of $P_X(\phi)$ in the $\Z^d$ setting for SI SFTs and locally constant and rational valued $\phi$, and show that in the special case $X = A^{\mathbb{Z}^2}$, this algorithm runs in singly exponential time.


\end{abstract}


\section{Introduction}

In this paper, we will explore computability properties of a few key objects in the thermodynamic formalism in the setting of symbolically defined topological dynamical systems called \textbf{subshifts} on countable amenable groups. A $G$-subshift (for 
a countable amenable group $G$) is defined by a finite alphabet $\mathcal{A}$ and a subset $X \subset \mathcal{A}^G$ which is closed (in the product topology) and invariant under the left-shift action $(\sigma_g)$ defined by $(\sigma_g x)(h) = x(hg)$ for all $x \in X$ and $g, h \in G$. When $G$ is clear from context, we will refer to a $G$-subshift as just a subshift.

 Given a subshift $X$ and any continuous map $\phi$ (called a \textbf{potential}) from $X$ to $\mathbb{R}$, one can define sums over finite patterns appearing in $X$, weighted according to $\phi$, called \textbf{partition functions}; the exponential growth rate of these partition functions (as a function of volume of the shape of the finite patterns) is called \textbf{topological pressure}. Topological pressure has a long history in statistical physics, related to concepts such as phase transitions, and is also a generalization of the long-studied concept of topological entropy in dynamical systems.


\subsection{Pressure for $\mathbb{Z}$-subshifts}

The most classical case is when the underlying group is $\mathbb{Z}$, meaning that the subshift $X$ consists of bi-infinite sequences from the alphabet. One of the earliest well-studied settings is when a $\mathbb{Z}$-subshift $X$ and potential $\phi$ are defined by local rules, meaning $X$ is defined via a finite list of forbidden words (such $X$ are called {\bf subshifts of finite type} or \textbf{SFTs}) and $\phi$ depends on only finitely many letters of its input sequence (such $\phi$ are called \textbf{locally constant}). In this case, the classical Perron-Frobenius theorem implies that the pressure is the logarithm of an algebraic number derivable directly from $X$ and $\phi$. 

More generally, there is no hope for such a closed form. In fact, if one fixes $\phi = 0$ (in which case the pressure is the \textbf{topological entropy}), then for every real $\beta > 0$ there is a $\mathbb{Z}$-subshift $X$ called a $\beta$-shift with pressure/entropy equal to $\beta$. Similar, if one fixes $X$ to be the full shift $\mathcal{A}^\mathbb{Z}$ over an alphabet $\mathcal{A}$, then pressure over the class of constant potentials $\phi = C$ is continuous in $C$, and so can achieve every possible real. Therefore, one needs to look for weaker conclusions, and it turns out that for more general groups, a natural notion is that of computability of the pressure map. 

Formal definitions are in Section~\ref{prelim}, but roughly speaking, a function $f$ is \textbf{computable} if there is a Turing machine which, on input $x$ and $n$, can output a rational approximation to within $2^{-n}$ of $f(x)$. It's important to note that general $\phi$ and $X$ (when not locally constant/an SFT respectively) cannot be expressed via a finite amount of information, and so when one discusses computability of pressure, it's usually understood that one has access to suitable representations or \textbf{oracles} for $\phi, X$ which can be read by an algorithm. Generally, $X$ is described via an enumeration of finite words which appear (called the \textbf{language}) and $\phi$ is described via a sequence of locally constant 
rational-valued approximations which converge to it uniformly. 

Some early results about computability of pressure were proved by Spandl, such as the following. (We do not define sofic shifts in this work, but they are a class containing SFTs.)

\begin{theorem}[\cite{sft-pressure-comp}, \cite{sofic-pressure-comp}]
If $X$ is a $\mathbb{Z}$-SFT or sofic shift, then the function $P_X$ is computable.
\end{theorem}

More recently, in \cite{pressure-comp-beyond-sft}, Burr, Das, Wolf, and Yang proved computability of $P_X$ for a large class of so-called coded subshifts, including many well-studied examples such as $S$-gap shifts and $\beta$-shifts. The authors additionally showed that pressure is actually not computable if one views it as a function of both $X$ and $\phi$ (rather than fixing $X$ and then taking $\phi$ as the variable). We do not yet define equilibrium states and entropy, but note that the hypotheses of the theorem are satisfied, for instance, by any mixing $\mathbb{Z}$-SFT and H\"{o}lder potential $\phi$. 

\begin{theorem}[Theorem B, \cite{pressure-comp-beyond-sft}] If $X$ is a $\mathbb{Z}$-subshift with alphabet $\mathcal{A}$, $\phi \in C(X)$, and all equilibrium states for $\phi$ on $X$ have positive entropy, then the generalized pressure function 
$P : (\Sigma \times C(\mathcal{A}^\Z)) \rightarrow \R$ (where $\Sigma$ represents the space of all $\mathbb{Z}$-subshift with alphabet $\mathcal{A}$) defined by $(X, \phi) \mapsto P_{X}(\phi)$
is not computable at $(X, \phi)$. 
\end{theorem}

The proof of lack of computability in Theorem B from \cite{pressure-comp-beyond-sft} relies on the somewhat surprising fact that pressure does not vary continuously as a function of $X$, even for $\phi = 0$. Therefore, for computability results moving forward, we generally assume that $X$ is a fixed subshift and that $P_X(\phi)$ is viewed only as a function of $\phi$. 

\subsection{Pressure for subshifts on more general acting groups}

The most natural class of groups to work with for results about presure are the countable amenable groups, since those are the ones on which the most meaningful theory of entropy/pressure has been defined. We postpone a formal definition to Section~\ref{prelim}, but informally, a countable group is amenable if it has a so-called F\o lner sequence, which is a sequence of finite subsets whose `boundaries have volume growing much more slowly than the sets themselves.' For our purposes, we also assume that $G$ is finitely generated and has decidable word problem, so that elements of $G$ have a representation suitable for computation.

One of the first surprising discoveries for more general groups is that the class of SFTs, which had such simple behavior for 
$\mathbb{Z}$, has some surprising pathologies for more general groups. A first discovery in this area (\cite{berger}) is that even for $\mathbb{Z}^2$, there is no algorithm which, given an alphabet and finite list of forbidden patterns, decides whether the associated SFT is nonempty (though there is a semi-algorithm). Later, in foundational work of Hochman and Meyerovitch, they proved that there are exist $\mathbb{Z}^2$-SFTs whose topological entropy (pressure for $\phi = 0$) is not computable. In fact, they proved that the class of entropies attained by SFTs is exactly those nonnegative reals which are \textbf{computable from above}, meaning that there is a Turing machine which outputs approximations from above, but without any knowledge of rate of convergence. (In particular, computability from above is far weaker than computability.)

\begin{theorem}[\cite{hoch-meyer}]
A nonnegative real $\alpha$ is the entropy of some $\mathbb{Z}^2$-SFT if and only if it is computable from above.
\end{theorem}

In fact, in \cite{amenable-entropy}, Barbieri extended this result to a much larger class of groups.

\begin{theorem}[\cite{amenable-entropy}]
If $G$ admits a translation-like action by $\mathbb{Z}^2$, then a nonnegative real $\alpha$ is the entropy of some 
$G$-SFT if and only if it is computable from above.
\end{theorem}

Since even $\phi = 0$ can yield noncomputable entropy for $\mathbb{Z}^2$-SFTs, for general groups, it seems the best one can hope for is that the pressure function is computable from above (see Section~\ref{prelim} for a definition). Our first main result shows that for finitely generated amenable $G$ with decidable word problem, given an enumeration for a forbidden list of $X$, pressure is always computable from above.

\begin{Theorem1} Let $G$ be a finitely generated, amenable group with decidable word problem and let $X \subset \mathcal{A}^G$ be any subshift. Then $P_X : C(\mathcal{A}^G) \rightarrow \R$ is computable from above given an enumeration for a forbidden list for $X$. 
\end{Theorem1}

In order to hope for true computability, more assumptions on $X$ must be made. One often useful type of hypothesis are so-called mixing properties, meaning that any two finite legal patterns can co-exist in a point of $X$ given large enough separation. Even these properties are often not enough to achieve computability even for $\phi = 0$, i.e. topological entropy (see \cite{gangloff-menibus}). However, Hochman and Meyerovitch (\cite{hoch-meyer}) showed that the strongest such studied property, \textbf{strong irreducibility (or SI)}, where the required separation between patterns is independent of the patterns chosen, has such implications for entropy.

\begin{theorem}[\cite{hoch-meyer}]
For any $d \geq 2$ and any strongly irreducible $\mathbb{Z}^d$-SFT $X$, the entropy of $X$ is computable.
\end{theorem}

This was generalized in \cite{marcus-pavlov} to pressure for locally constant potentials taking computable values.  
Our second main result shows that this also holds for more general potentials and groups.

\begin{Theorem2} Let $G$ be a finitely generated, amenable group with decidable word problem. Let $X \subset \mathcal{A}^G$ be a strongly irreducible SFT. 
Then the pressure function $P_X : C(\mathcal{A}^G) \rightarrow \R$ is computable. 
\end{Theorem2}

A key to the arguments of \cite{hoch-meyer} is proving that for an SI $\mathbb{Z}^d$-SFT, it's algorithmically decidable which finite patterns appear in points of $X$. It was shown in \cite{pavlov-schraudner} that if $X$ has a point with finite orbit, then one can give upper bounds on computation time, and such a point is guaranteed for $d = 2$. However, it's unknown for $d \geq 3$ whether SI $\mathbb{Z}^d$-SFTs contain such points, and in this case there are no known bounds on computation time of the language. However, when an oracle for the language of $X$ is given, and $\phi$ is a rational valued locally constant potential, we can provide an upper bound on the computation time.

\begin{Theorem3}
For any $d$, there exists a Turing machine $T$ that, upon input of: 
\begin{itemize} 
\item an SI SFT $X \subset \mathcal{A}^{\Z^d}$, 
\item an oracle for the language of $X$, and
\item a locally constant potential $\phi$,
\end{itemize} computes $P_X(\phi)$ to precision $2^{-k}$ in $|\mathcal{A}|^{O(2^{kd})}$ time. 
\end{Theorem3}

In fact, when $X$ is a full shift $\mathcal{A}^{\mathbb{Z}^d}$, the computation time can be reduced significantly.

\begin{Theorem4}
There exists a Turing machine $T$ that, upon input of a rational-valued locally constant potential $\phi$ on a full shift $\mathcal{A}^{\mathbb{Z}^2}$, computes $P_X(\phi)$ to precision $2^{-k}$ in $|\mathcal{A}|^{O(2^k)}$ time. 
\end{Theorem4}

In particular, this shows that pressure of locally constant potential on two-dimensional full shifts is always computable in singly exponential time. We have no results on computation time for general countable amenable $G$, since we do not know of any bounds on time required to generate a F\o lner sequence for $G$.

\subsection{Computability of zero-temperature limits}

As corollaries, we can also prove computability results about some properties of \textbf{ground states}, which are another concept from statistical physics. For any potential $\phi$ and subshift $X$, it turns out that the topological pressure is the supremum, over all shift-invariant Borel probability measures $\mu$, of $h(\mu) + \int \phi \ d\mu$, where $h(\mu)$ is measure-theoretic entropy (see \cite{walters} for a definition). For $G$-subshifts, this supremum is always achieved, and any measure $\mu$ achieving the supremum is called an \textbf{equilibrium state}. 
For a fixed $\phi$, any weak-* limit point of equilibrium states for $\beta \phi$ for $\beta \rightarrow \infty$ is called a \textbf{ground state} for $\phi$. (The parameter $\beta$ has a physical interpretation as the multiplicative inverse of temperature, and so ground states are often called zero-temperature limits.)

It's easily shown that all ground states must maximize $\int \phi \ d\mu$ among shift-invariant $\mu$, and must achieve the maximal possible entropy among that collection. We call the common entropy of all ground states the \textbf{ground state entropy}, and the common integral of $\phi$ among ground states the \textbf{ground state energy}. 

While there has been work done on characterizing computability properties of ground states themselves, such as in  \cite{ground-state-computable}, we will consider the functions mapping potentials to their ground state entropy and energy. In \cite{zero-temp-computable}, Burr and Wolf proved that even for $\mathbb{Z}$-SFTs, the ground state entropy function is not computable in general, though it is computable from above. We extend this result to more general groups, and also show that in contrast, the ground state energy is computable when the pressure function is. 


\begin{Theorem5} Let $G$ be a countable, amenable group and let $X \subset \mathcal{A}^G$ be a subshift such that $P_X$ is a computable function. Then, the function that sends $\phi$ to its ground state energy 
$\sup_{\nu \in M_\sigma(X)} \int \phi \ d\nu$ is computable. 
\end{Theorem5}

Notice in particular this applies to subshifts $X$ satisfying the hypotheses of Theorem \ref{pressure is computable}, i.e. when $X$ is a strongly irreducible SFT and $G$ is finitely generated with decidable word problem. Similarly, we obtain the following corollary: 

\begin{Cor1}  Let $G$ be a countable, amenable group and let $X \subset \mathcal{A}^G$ be a subshift such that $P_X$ is computable from above. Then the function that sends $\phi$ to its ground state energy $\sup_{\nu \in M_\sigma(X)} \int \phi \ d\mu$ is computable from above. 
\end{Cor1}

Further note that this applies to subshifts $X$ satisfying the hypotheses of Theorem \ref{pressure is computable from above}, i.e. ground state energy is always computable from above given an enumeration for a forbidden list of $X$ when $G$ is finitely generated with decidable word problem. We finally use Theorem \ref{ground state energy} to obtain our final result: 

\begin{Cor2} Let $G$ be a countable, amenable group and let  $X \subset \mathcal{A}^G$ be a subshift such that $P_X$ is computable. Then, the function that sends $\phi$ to the ground state entropy of $\phi$ is computable from above. 
\end{Cor2}

Again we note that this applies to any subshift satisfying the conditions of Theorem \ref{pressure is computable} including strongly irreducible SFTs when $G$ is finitely generated with decidable word problem, and a large class of subshifts when $G= \Z$ including SFTs, sofic shifts, coded subshifts, and many other well studied examples.  We note that there is no hope of extending Corollary~\ref{ground state entropy} to guarantee computability; we give a simple example of a locally constant $\phi$ on a $\mathbb{Z}^2$ full shift for which the ground state entropy is not computable.

The structure of the paper is as follows. We begin Section 2 with the relevant preliminaries, discussing relevant topics from computability theory and introducing finitely generated amenable groups. We then formally define subshifts and pressure, discussing relevant computability properties throughout. 

In Section 3 we prove Theorem \ref{pressure is computable}, using on so-called quasi-tilings studied in \cite{ornsteinweiss} and \cite{dhz} to create our algorithm for computing pressure.
Section 4 uses similar techniques to prove Theorem \ref{pressure is computable from above}. 
In Section 5 we then move to ground states, proving Theorem \ref{ground state energy} and Corollaries \ref{computable from above cor ground state} and \ref{ground state entropy}. We conclude with Section 6, where we prove upper bounds on computation time for computation of pressure of a locally constant potential on an SI $\mathbb{Z}^d$-SFT.

\section{Preliminaries}\label{prelim}

\subsection{Computability Theory} 

For a comprehensive introduction to computability theory we refer the reader to \cite{tutorial-computability} and \cite{comp-analysis-book}. In this paper, we will be interested in computability properties of certain numbers and functions. We will begin by outlining relevant computability definitions in the case of real numbers and real valued functions, and then move into the more general computable metric space setting. 

For any $\alpha \in \R$, we say an {\bf oracle for $\alpha$} is a function $\gamma : \N \rightarrow \Q$ such that for all $n \in \N$, $|\gamma(n) - \alpha| < 2^{-n}$. We say $\alpha \in \R$ is {\bf computable} if there exists a Turing machine that is an oracle for $\alpha$. We will additionally say that $\alpha \in \R$ is {\bf computable from above} if there exists a Turing machine $T$ such that $(T(n))_{n \in \N}$ is a sequence of rationals decreasing to $\alpha$. Notice that we have no control over the rate of convergence in this case. A function $f: \R \rightarrow \R$ is {\bf computable} if there exists a Turing machine $T$ such that for every $x \in \R$ and for any oracle $\gamma$ for $x$, $|T(\gamma, n) - f(x)| < 2^{-n}$. 

We now let $X$ be a separable metric space with metric $d$ and let $(s_n)$ be a dense sequence of points in $X$. We say that $(X, d, (s_n))$ is {\bf a computable metric space} if there exists a Turing machine $T$ such that for all $i, j, n \in \N$, $|T(i, j, n) - d(s_i, s_j)| < 2^{-n}$. In this context, $(s_n)$ acts as our accessible dense subset in $X$, similar to $\Q$ in $\R$. With this in mind, it is natural to say for any $x \in X$, an oracle for $x$ is a function $\gamma: \N \rightarrow \N$ such that $d(x, s_{\gamma(n)}) < 2^{-n}$. Similar to above we say $x \in X$ is computable (with respect to $d$ and $(s_n)$) if there exists a Turing machine $T$ that is an oracle for $x$. We now let $(X, d, (s_n))$ be a computable metric space. We say $f: X \rightarrow \R$ is computable if there exists a Turing machine $T$ such that for any $x \in X$, and any oracle $\gamma$ for $x$, for all $n \in \N$, $|T(\gamma, n) - f(x)| < 2^{-n}$. 
Similarly, $f$ is computable from above if, for any $x$ and oracle $\gamma$ for $x$, the sequence $(T(\gamma, n))$ decreases to $f(x)$.




For a given subshift $X$, the space of potentials is $C(X)$, the set of all real-valued continuous functions on $X$ equipped with the uniform norm. Since we will require that the domain of our pressure function is a computable metric space, we will consider the domain of $P_X$ to be $C(\mathcal{A}^G)$. We will take $(s_n)$ to be any computable enumeration of rational-valued functions in $C(\mathcal{A}^G)$ which are \textbf{locally constant}, meaning that the output for an input $x$ depends on only finitely many letters of $x$. It is not hard to show that there exists a Turing machine that approximates distances in the uniform norm topology between rational-valued locally constant potentials, and so $(C(\mathcal{A}^G), d_{|| \cdot ||_\infty}, (s_n))$ is a computable metric space. It is through these locally constant functions that we will determine the computability properties of $P_X : C(\mathcal{A}^G) \rightarrow \R$.


\subsection{Countable Amenable Groups}

We let $G$ be a countable group and denote the identity of $G$ by $e$. We will use the notation $K \Subset G$ to indicate that $K$ is a finite subset of $G$. A (left) \textbf{F\o lner sequence} is a sequence $F_1 \subsetneq F_2 \subsetneq F_3 \subsetneq \cdots$ of finite subsets of $G$ such that $\bigcup_{n \in \N} F_n = G$ and for any $K \Subset G$, $\lim_{n \rightarrow \infty} |K F_n  \triangle F_n| / |F_n| = 0$. 
For any $\epsilon > 0$ and $F, K \Subset G$, we say $F$ is {\bf $(K, \epsilon)$-invariant} if $|KF \triangle F| / |F| < \epsilon$. 
A countable group $G$ is \textbf{amenable} if it admits a F\o lner sequence or, equivalently, if for every $\epsilon > 0$ and $K \subset G$, there exists a $(K, \epsilon)$-invariant set.

From now on, we assume that $G$ is finitely generated. We say that $G$ has \textbf{decidable word problem} if there exists a Turing machine that, upon input of a finite string of generators (e.g. $g_1 g_3 g_4^{-1} g_1$), returns whether or not the string was equal to the identity element $e$. Examples of finitely generated, amenable groups with decidable word problem include 
$\Z^d$, Heisenberg groups, and Baumslag-Solitar groups. For the remainder of the paper we will assume our group $G$ has decidable word problem. For completeness we give a proof for the following well-known fact (proved, for instance, in \cite{cavaleri}): 

\begin{lemma}\label{computable folner} Let $G$ be a finitely generated, amenable group with decidable word problem. Then there exists a Turing machine $T$ such that upon input $n$, returns $T(n) \Subset G$ such that $\{ T(n) \}$ forms a F\o lner sequence in $G$. 
\end{lemma}

\begin{proof} We first fix an enumeration of the finite subsets of $G$,  $\{ E_k : k \in \N \}$ which can be done in a computable manner since $G$ has decidable word problem by assumption. We let $T(1) = \{ e \}$ and inductively define $T(n)$ for all $n > 1$ as follows. 

Fix $i = 1$, test the following conditions: 
\begin{itemize}
\item $T(n-1), \bigcup_{k=1}^{n-1} E_k \subsetneq E_i$, 
\item For all $1 \leq k \leq n$, $| E_k E_i \triangle E_i| / |E_i| < 1/n$. 
\end{itemize}
All of these conditions are testable in finite time since $G$ has decidable word problem. If any of these conditions fail, increment $i$ and repeat the above checks. Since $G$ is a countable amenable group, we know the conditions will hold for some $i$, at which point we return $T(n) = E_i$. 

Clearly by construction $\{ T(n) \}$ forms a F\o lner sequence in $G$. 
\end{proof}

Unfortunately in general the above approach does not provide us with any control over the growth rate of the sets in the F\o lner sequence, or computation time. However, many countable amenable groups have extremely simple 
F\o lner sequences. For example, in $\Z^d$, the sets $[-n, n]^d$ form a F\o lner sequence such that every element of the sequence exactly tiles $\Z^d$ by translates.

As an extension of this concept of F\o lner sets tiling $G$, we will later be taking advantage of the proof techniques used in \cite {dhz} by Downarowicz, Huczek, and Zhang regarding exact tilings of an amenable group. We will define three kinds of tilings that will be of use in this paper, all originally due to Ornstein and Weiss in \cite{ornsteinweiss}. 

For a collection $\mathcal{T}$ of finite subsets of $G$ and for $\alpha \in [0, 1]$, we say $\mathcal{T}$ {\bf $(1-\alpha)$-covers} $G$ if for any F\o lner sequence we have 
$$\liminf_{n \rightarrow \infty} \frac{ | \bigcup_{T \in \mathcal{T}} T \cap F_n|}{|F_n|} \geq 1-\alpha. $$
We will also say that $\mathcal{T}$ is {\bf $\alpha$-disjoint } if there exists a map $T \mapsto T^\circ \subset T$ $(T \in \mathcal{T})$, such that 
\begin{itemize}
\item For distinct $T, S \in \mathcal{T}$, $T^\circ \cap S^\circ = \emptyset$, and 
\item For $T \in \mathcal{T}$, $|T^\circ|/|T| > 1-\alpha$.
\end{itemize}
In other words, $\mathcal{T}$ is $\alpha$-disjoint if we can remove at most $\alpha$ proportion of each tile to make the collection disjoint. 

We now say that a collection of finite subsets of $G$, $\mathcal{T} = \{ T_i : 1 \leq i \leq k \}$, is a {\bf quasitiling} if there exists a corresponding collection of finite shapes $\mathcal{S} = \{ S_i : 1 \leq i \leq k \}$ s.t. for every $T \in \mathcal{T}$, there exist $g \in G$ and $s \in \mathcal{S}$ s.t. $T = Sg$. We say that the quasitiling $\mathcal{T}$ is {\bf induced by} $\mathcal{S}$.

We now say that a collection of disjoint finite tiles $\mathcal{T}$ is a {\bf $(1-\alpha)$-tiling} if it is a $(1-\alpha)$-covering quasitiling. Finally, a collection of disjoint finite tiles $\mathcal{T}$ is an {\bf exact tiling} if $G = \bigcup_{T \in \mathcal{T}}T$. We now note an important result from \cite{dhz}:



\begin{fact}[Theorem 4.3 \cite{dhz}] Given any infinite, countable amenable group $G$, any $\epsilon > 0$, and any finite $K \Subset G$, there exists an exact tiling $\mathcal{T}$ where each shape is $(K, \epsilon)$-invariant. 
\end{fact}

In fact, their result included statements regarding the entropy of the tiling space which we have omitted since they are not necessary for the purposes of this paper. Additionally, we will not be using the full strength of their theorem, as we do not need exact tilings for our purposes. The main utility for us is that their construction is clearly algorithmically computable. 

\subsection{Subshifts} 

In this paper, we will concern ourselves with subshifts over a finite alphabet. Let $\mathcal{A}$ be a finite set with the discrete topology and let $G$ be a finitely generated, amenable group with decidable word problem. We endow $\mathcal{A}^G$ with the product topology - this will be our ambient space. The dynamics will be induced by the {\bf left shift map} $\sigma$ where for all $g , h \in G$, for all $x \in \mathcal{A}^G$, $\sigma_gx(h) = x(hg)$. We call a closed, $\sigma$-invariant $X \subset \mathcal{A}^G$ a {\bf subshift}. For any $H \subset G$, we will denote the restriction of $x$ to $H$ by $x|_H$. We say that a finite word $w \in A^F$ is a {\bf subword} of a (finite or infinite) word $v \in A^H$ if there exists $g \in G$ so that $Fg \subset H$ and $\forall f \in F$, $w(f) = v(fg)$.

For any $H \Subset G$, we define the {\bf $H$-language} of the subshift $X$ by $L_H(X) = \{ x|_H : x \in X \}$. The {\bf language of $X$} will be $L(X) = \bigcup_{H \Subset G} L_H(X)$. Denote $\mathcal{A}^* = \bigcup_{H \Subset G} \mathcal{A}^H$. For a collection $\mathcal{F} \subset \mathcal{A}^*$ of {\bf forbidden words}, we denote the subshift induced by 
$\mathcal{F}$ by $X_{\mathcal{F}} = \{ x \in \mathcal{A}^G : \textrm{ no } w \in \mathcal{F} \textrm{ is a subword of } x \}$. 
It is easy to see that for any subshift $X$, if we let $\mathcal{F} = \mathcal{A}^* \backslash L(X)$ then $X = X_{\mathcal{F}}$. Indeed, such an $\mathcal{F}$ need not be unique.  

We say a subshift $X$ is a {\bf subshift of finite type} (or SFT) if there exists a finite $\mathcal{F} \Subset \mathcal{A}^*$ such that $X = X_{\mathcal{F}}$. Additionally, we say an SFT $X$ is {\bf strongly irreducible} (or SI) if there exists a universal $E \Subset G$ such that for any $F, H \Subset G$ with $v \in L_F(X)$, $w \in L_H(X)$, if $FE \cap H = \emptyset$, then $vw \in L_{F \cup H}(X)$. 

A basis for the product topology on $\mathcal{A}^G$ can be taken in the from of cylinder sets. For any  $F \Subset G$, $v \in \mathcal{A}^F$, define the {\bf cylinder set} of $v$ to be 
$$[v] = \{ x \in \mathcal{A}^G : x_F = v\}. $$
For any subshift $X$, the standard subspace topology has a basis of cylinder sets intersected with $X$. 

For a given subshift $X \subset \mathcal{A}^G$, we denote $M_\sigma(X)$ to be the set of all $\sigma$-invariant, Borel probability measures on $X$. A measure $\mu$ is $\sigma$-invariant if for all measurable $A \subset X$, for all $g \in G$, $\mu(A) = \mu(\sigma_{g^{-1}}A)$. We note here that $M_\sigma(X)$ is compact in the weak-* topology and non-empty for any subshift $X \subset \mathcal{A}^G$.

\subsection{Decidability of Language} 

For a subshift $X$, a natural question is: given a shape $E \Subset G$ and a configuration $v \in \mathcal{A}^E$, is $v \in L(X)$? With this in mind, we say a subshift $X$ has {\bf decidable language} if there exists a Turing machine $T$ such that for any given shape $E \Subset G$, and $v \in \mathcal{A}^E$, upon input $v$, $T$ halts in finite time returning whether or not $v \in L(X)$.  The proof of the following theorem is adapted from that of Hochman and Meyerovitch in \cite{hoch-meyer} which applies to the $G=\Z^d$ case. 

\begin{fact}\label{decidable language} Let $G$ be finitely generated amenable with decidable word problem, $X \subset \mathcal{A}^G$ be an SI SFT with finite forbidden list $\mathcal{F}$. 
Then $X$ has decidable language. 
\end{fact}

Take $\mathcal{F}$ a forbidden list for the SI SFT $X$, and assume without loss of generality that $\mathcal{F} \subset A^E$ where $E$ satisfies SI.
For each $F \Subset G$ and $v \in \mathcal{A}^F$, we say $v$ is {\bf locally admissible} (according to $\mathcal{F}$) if for every $w \in \mathcal{F}$, $w$ is not a subword of $v$. Notice that since $\mathcal{F}$ is finite, it is decidable whether or not $v$ is locally admissible for any $v \in \mathcal{A}^F$. Additionally, we say for $v \in \mathcal{A}^F$, $v$ is {\bf globally admissible} if $v \in L(X)$. For the remainder of the proof we assume without loss of generality (by passing to a subsequence if necessary) that our F\o lner sequence satisfies for each $n \in \N$, $F_n E \subset F_{n+1}$. 

\begin{defn} For $n < N$, locally admissible $a \in \mathcal{A}^{F_n}$ and $b \in \mathcal{A}^{F_N}$, we say that $a$ and $b$ are {\bf compatible} if there exists a locally admissible $c \in \mathcal{A}^{F_N}$ such that $c |_{F_n} = a$ and $c|_{F_N \backslash F_{N-1}} = b_{F_N \backslash F_{N-1}}$. 
\end{defn}

We define the following conditions for any $n \in \N$, $a \in \mathcal{A}^{F_n}$ and $N \geq n$. 
\begin{enumerate}
\item[(i)] $a \neq b|_{F_n}$ for every locally admissible $b \in \mathcal{A}^{F_N}$, 
\item[(ii)] $a$ and $b$ are compatible for every locally admissible $b \in \mathcal{A}^{F_N}$. 
\end{enumerate}

\begin{lemma} For any $a \in \mathcal{A}^{F_n}$, the following hold: 
\begin{itemize}
\item (a) $a$ fails to be globally admissible if and only if (i) holds for some $N > n$ 
\item (b) $a$ is globally admissible if and only if (ii) holds for some $N > n$ 
\end{itemize}
\end{lemma}


\begin{proof} (a) First suppose $a$ is globally admissible. Then for each $N > n$, $a =  b|_{F_n}$ for some globally admissible, and therefore locally admissible, $b \in \mathcal{A}^{F_N}$. In particular, (i) fails to hold for all $N > n$. We now suppose that (i) fails to hold for all $N \geq n$. Thus for all $N > n$, there exists some locally admissible $b \in \mathcal{A}^{F_N}$ such that $a = b|_{F_n}$. By compactness it follows that $a$ is globally admissible.

(b) First we notice that if (ii) holds for $N$, then (ii) holds for all $M \geq N$. To show this, let $a \in \mathcal{A}^{F_n}$ be locally admissible and suppose (ii) holds for some $N > n$. We now let $M \geq N$ and let $b \in \mathcal{A}^{F_M}$ be locally admissible. Then $b|_{F_N}$ is locally admissible and by assumption we have some $c \in \mathcal{A}^{F_N}$ such that $c|_{F_n} = a$ and $c|_{F_N \backslash F_{N-1}} = b|_{F_N \backslash F_{N-1}}$.  Since by assumption for all $n$ $F_nE \subset F_{n+1}$, since $b$ was locally admissible, we know $c b|_{F_M \backslash F_N}$ is locally admissible, and thus (ii) holds for $M$.

We now suppose that for some $N > n$ and for all $b \in \mathcal{A}^{F_N}$, $a$ and $b$ are compatible. For each $M \geq N$ we therefore have some locally admissible $c_M \in \mathcal{A}^{F_M}$ such that $c_M|_{F_n} = a$ and by compactness it must be the case that $a$ is globally admissible.

We now suppose that $a$ is globally admissible. Since $E$ satisfies SI for $X$ and by assumption $F_n E \subset F_{n+1}$ for all $n \in \N$, we know for all globally admissible $b \in \mathcal{A}^{F_{n+2}}$, $a$ and $b$ are compatible. In particular, since $a$ and $b$ are globally admissible, we know that $a \in L(X)$ and $b|_{F_{n+2} \backslash F_{n+1}} \in L(X)$. Since by assumption $F_n E \subset F_{n+1}$, we have $F_n E \cap ( F_{n+2} \backslash F_{n+1}) = \emptyset$, and we therefore know $a b|_{F_{n+2} \backslash F_{n+1}} \in L(X)$. Thus by definition, $a$ and $b$ are compatible. 

It immediately follows that, for every locally admissible $b \in \mathcal{A}^{F_{n+2}}$ if $a$ and $b$ are not compatible, then $b$ is not globally admissible. Since there are finitely many $b \in \mathcal{A}^{F_{n+2}}$, by application of (a), there exists some $N \in \N$ such that for all locally admissible but not globally admissible $b \in \mathcal{A}^{F_{n+2}}$, $b$ is not compatible with any locally admissible $c \in \mathcal{A}^{F_N}$. We therefore have for all locally admissible $c \in \mathcal{A}^{F_N}$, there exists a globally admissible $b \in \mathcal{A}^{F_{n+2}}$ such that $b$ and $c$ are compatible.

We now let $c \in \mathcal{A}^{F_N}$ be locally admissible. We therefore have some globally admissible $b \in \mathcal{A}^{F_{n+2}}$ such that $b$ and $c$ are compatible. Let $c' \in \mathcal{A}^{F_N}$ witness compatibility of $b$ and $c$ and notice in particular that $c'|_{F_N \backslash F_{N-1}} = c|_{F_N \backslash F_{N-1}}$. To show that $a$ and $c$ are compatible, it is therefore sufficient to show that $a$ and $c'$ are compatible. Since $b$ is globally admissible, we know that $a$ and $b$ are compatible and we let $b' \in \mathcal{A}^{F_{n+2}}$ witness compatibility of $a$ and $b$. We now claim that $d = b' c'|_{F_N \backslash F_{n+2}}$ is locally admissible. 

First notice that to check if $d$ is locally admissible, since $E$ satisfies the SFT rules for $X$ it is sufficient to check that for any shift of $E$, $Eg$, $d|_{Eg} \notin \mathcal{F}$. Since $b'$ is locally admissible, for any $Eg \subset F_{n+2}$, $d|_{Eg} \notin \mathcal{F}$. Additionally, for any $Eg $ not contained in $F_{n+2}$, we know $d|_{Eg} = c'|_{Eg}$, which is locally admissible and in particular $d|_{Eg} \notin \mathcal{F}$. We have therefore identified a locally admissible $d \in \mathcal{A}^{F_N}$ such that $d|_{F_n} = a$ and $d|_{F_N \backslash F_{N-1}} = c|_{F_N \backslash F_{N-1}}$. By definition we therefore know $a$ and $c$ are compatible.

\end{proof}

Theorem \ref{decidable language} follows immediately by application of the following algorithm. For any $F \Subset G$, let $v \in \mathcal{A}^F$. First identify the smallest $n$ such that $F \subset F_n$. We therefore have finitely many locally admissible $a \in \mathcal{A}^{F_n}$ such that $v$ is a subword of $a$. For each such $a \in \mathcal{A}^{F_n}$ perform the following algorithm: For each $N \geq 1$, find the first $N$ such that (i) or (ii) holds in the lemma (noticing that for each $N$,  these conditions are finitely checkable). If (ii) holds then $a$ is globally admissible, otherwise if (i) holds it is not. It follows then that $v$ is globally admissible and if and only if there exists some globally admissible $a \in \mathcal{A}^{F_n}$ containing it, and we can therefore decide in finite time if $v \in L(X)$. 

Notice here that in the $\Z^d$ case, when the SI assumption is dropped not only does $X$ not necessarily have decidable language, as shown by Berger in \cite{berger}, it is not necessarily decidable whether or not a given finite $\mathcal{F}$ generates a non-empty SFT $X_\mathcal{F}$.

\subsection{Partition Functions, Pressure, and Ground States}

We can now define topological pressure. For a given subshift $X$, we call a continuous, real valued $\phi \in C(X)$ a {\bf potential}. Since every $\phi \in C(X)$ can be extended to a potential on $\mathcal{A}^G$, which has a suitable structure for computation, we will be considering the full set of potentials $C(\mathcal{A}^G)$, restricting to $X$ whenever necessary.

For any finite $F \Subset G$, we define the {\bf partition function} of $F$ to be: 
$$Z_F(\phi) = \sum_{w \in L_F(X)} exp \left( \sup_{x \in [w]} \sum_{g \in F} \phi ( \sigma_gx) \right). $$
The {\bf topological pressure} of a potential $\phi$ is then defined by 
$$P_X(\phi) = \lim_{n \rightarrow \infty} \frac{1}{|F_n|} \log Z_{F_n}(\phi). $$

(The existence of the limit is non-trivial; see for instance \cite{ollagnier} or \cite{ornsteinweiss} for more details.)
This paper will be primarily concerned with the computability of the function $P_X : C(\mathcal{A}^G) \rightarrow \R$. Notice here that $P_X(\phi)$ does not depend on choice of F\o lner sequence, and in fact as noted in Remark 2 of \cite{countable-thermoformalism}, pressure satisfies the {\bf infimum rule}: for any $\phi \in C(\mathcal{A}^G)$, 
$$P_X(\phi) = \inf_{F \Subset G, F \neq \varnothing} \frac{1}{|F|} \log Z_F(\phi). $$
The infimum rule will be essential to many of our computations later. We note that pressure for the zero potential, $P_X(0) = h(X)$ is the {\bf topological entropy} of the subshift. 


Pressure also has a representation as a supremum, called the Variational Principle. In particular, for any $\phi \in C(\mathcal{A}^G)$, 
$$P_X(\phi) = \sup_{\mu \in M_\sigma(X)} h(\mu) + \int \phi \ d\mu, $$
where $h(\mu)$ is the measure-theoretic (Kolmogorov-Sinai) entropy of $\mu$ (again, see \cite{walters} for more information). 
We say $\mu$ is an {\bf equilibrium state} for $\phi$ if it attains the supremum: $P_X(\phi) = h(\mu) + \int \phi d\mu$.  In the subshift setting $h$ is upper semi-continuous in the weak-$*$ topology (and $\int \phi \ d\mu$ is continuous by definition), guaranteeing the existence of at least one equilibrium state for any potential. 


In addition to understanding the pressure of arbitrary potentials, it is also of interest to understand, for a fixed potential $\phi$,  the map $P_{X,\phi} : (0, \infty) \rightarrow \R$ where $P_{X,\phi}(\beta) := P_X(\beta \phi)$. In particular, $\beta$ has a physical interpretation as the multiplicative inverse of temperature. We call any weak-* limit point of equilibrium states for $\beta \phi$ for $\beta \rightarrow \infty$ a \textbf{ground state} for $\phi$, or a {\bf zero-temperature limit}. 

By weak-* compactness of the set of invariant measures, it is clear that the collection of ground states is nonempty. Further, for any ground state $\mu$, it can easily be shown that $\mu$ maximizes entropy among $\phi$-maximizing invariant measures. 
It is therefore well-defined to refer to the {\bf ground state energy} of a potential $\phi$, denoting $\sup_{\nu \in M_\sigma(X)} \int \phi d\nu$. Additionally, we can unambiguously refer to the {\bf ground state entropy} (also referred to as {\bf residual entropy}) of a potential $\phi$, denoting the common value of entropy among ground states for $\phi$.

\section{$P_X$ is Computable for SI SFTs}

Our goal in this section is to prove Theorem~\ref{pressure is computable}, which we restate here for convenience.

\begin{theorem}\label{pressure is computable} Let $G$ be a finitely generated, amenable group with decidable word problem. Let $X \subset \mathcal{A}^G$ be an SI SFT. 
Then the pressure function $P_X : C(\mathcal{A}^G) \rightarrow \R$ is computable. 
\end{theorem}

We will prove the theorem through a sequence of lemmas. First we say that $\phi \in C(X)$ is {\bf locally constant} for a shape $E \Subset G$ if for all $x, y \in X$ such that $x|_E = y|_E$, $\phi(x) = \phi(y)$. We now let $\phi \in C(\mathcal{A}^G)$ and let $\gamma$ be an oracle for $\phi$. In this context, we mean that $\gamma(n) \in C(\mathcal{A}^G)$ is a rational-valued locally constant potential defined on $E_n \Subset G$ such that $|| \gamma(n) - \phi||_\infty < 2^{-n}$ and where the sequence $(E_n)$ is computable. 

We will now construct the algorithm to approximate $P_X(\phi)$ from $\gamma$. Fix some rational $\epsilon > 0$ and let $\varphi = \gamma(n)$ such that $2^{-n} < \epsilon / 2$. It's easily checked that $P_X$ is Lipschitz with constant $1$, and so $|P_X(\phi) - P_X(\varphi)| < \epsilon / 2$. It is therefore sufficient to approximate $P_X(\varphi)$ to $\epsilon / 2$ precision. By potentially expanding $E$, we may assume without loss of generality that $\varphi$ is locally constant for $E$. Additionally, it is well known that for any $c \in \R$, $P_X(\varphi + c) = P_X(\varphi) + c$. Since $\varphi$ is locally constant and rational valued, we can compute $||\varphi||_\infty$ in finite time and by adding $c = ||\varphi||_\infty$ we may assume for the purposes of our computations that $\varphi \geq 0$. For the remainder of this proof we fix this $\varphi$ and $E$. 

Let $\eta > 0$ be a rational number satisfying: 
$$\frac{1}{1-\eta/2} - (1-\eta) < \frac{\epsilon}{2 \left( \log  | \mathcal{A} | + ||\varphi ||_\infty \right) }. $$
We will now follow the construction in \cite{dhz} to construct a collection of shapes $\mathcal{S}$ that $(1-\eta/2)$-tile $G$ and for which each shape is $(E, \eta/2)$-invariant.  (We note that the existence of such a collection already follows from results in 
\cite{ornsteinweiss}, but we need to be able to algorithmically construct it, which is more easily done via arguments from \cite{dhz}.)

\subsection{Construction of the tileset}

We begin by noting that the algorithm described in the proof of Lemma \ref{computable folner} provides us with a computable process to construct a F\o lner sequence $\{ F_n \}$ in $G$. We fix this F\o lner sequence and notice, for any $F \Subset G$ and any rational $\xi > 0$, we can identify some $N \in \N$ such that for all $n \geq N$ $F_n$ is $(F, \xi)$-invariant. The following lemma follows almost immediately from the proof of Lemma 4.1 in \cite{dhz}. 

\begin{lemma} There exists a Turing machine $T$ which, on input a finite $F \Subset G$, rational $\xi > 0$, and $n_0 \in \N$, outputs 
$r(\xi) \in \N$ and $n_0 < n_1 < \dots < n_r$ such that  $\{ F_{n_1}, \dots, F_{n_r} \}$ 
is a tile set of $(F, \xi)$-invariant shapes inducing an $\xi$-disjoint, $(1-\xi)$-covering quasitiling of $G$. 
\end{lemma}

\begin{proof} As in \cite{dhz} we take $r \in \N$ such that $(1 - \frac{\xi}{2} )^r < \xi$. Since $\xi \in \Q$ this is clearly computable. Take $n_1 > n_0$ such that $F_{n_1}$ is $(F, \xi)$-invariant and $F \subset F_{n_1}$. Notice that since $G$ has decidable word problem, both of these conditions are checkable in finite time. For each $1 \leq i \leq r$ take $\delta_i$ sufficiently small such that 
$$\frac{1-(1-\frac{\xi}{2})^{r-i}}{(1+\delta_i)^2} - \frac{\delta_i}{1+\delta_i} + \xi \left( \frac{1}{1+\delta_i} - (1-\delta_i)\left(1-\left( 1-\frac{\xi}{2}\right)^{r-i} \right) \right) > 1 +\left(\frac{3 \xi}{4} - 1\right)  \left(1-\frac{3\xi}{8}\right)^{r-i} .  $$

We now inductively identify our $n_i$ for $1 < i \leq r$. Let $i > 1$ and take $n_i > n_j$ for all $1 \leq j < i$ such that $F_{n_i}$ is $(F_{n_j}, \delta_j)$-invariant for all $1 \leq j < i$. We therefore have a finite process for which we can identify our desired $n_0 < n_1 < \dots < n_r$ by application of the proof of Lemma 4.1 in \cite{dhz}. 
\end{proof}

Note in the above lemma we do not claim to compute the tiling, only that we can compute the shapes that will be used for such an $\xi$-disjoint, $(1-\xi)$-covering quasitiling.

\begin{lemma}\label{lemma33} (Lemma 3.3 \cite{dhz}) Let $F \Subset G$ be a finite set and $\xi > 0$. Then there exists a (computable from $\xi$) $\delta > 0$ such that if $T \subset G$ is $(F, \delta)$-invariant and $T'$ satisfies $\frac{|T' \triangle T|}{|T|} \leq \delta$, then $T'$ is $(F, \xi)$-invariant. 
\end{lemma}

\begin{proof} 
First note that if $T'$ satisfies the given hypothesis, then $|T'| \geq |T| - |T' \triangle T| \geq |T| (1 - \delta)$. Now note that
$$\frac{|T' F \triangle T'|}{|T'|} \leq 
\frac{|TF \triangle T| + |T' \triangle T| + |T'F \triangle FT|}{(1 - \delta)|T|}
\leq \frac{|F| \delta + 2 \delta}{1-\delta}. $$
Take $\delta > 0 $ satisfying $\frac{\delta}{1-\delta} < \frac{ \xi}{|F|+2}$ and we can conclude the result. 
\end{proof}

As above, the following lemma follows from the proof techniques used in Lemma 4.2 of \cite{dhz}. 

\begin{lemma} There exists a Turing machine $T$ which, on input any finite $F \Subset G$ and rational $\xi > 0$, outputs a finite collection of shapes $\mathcal{S}$ such that each shape is $(F, \xi)$-invariant and there exists a $(1-\xi)$-tiling of $G$ induced by $\mathcal{S}$. 
\end{lemma}

\begin{proof} Let $\lambda > 0$ be a rational number such that $(1-\lambda)^2 > 1-\xi$ and $\frac{\lambda}{1-\lambda} < \frac{ \xi}{|F|+2}$. We now take $r = r(\lambda)$ as defined in Lemma 4.1 \cite{dhz}; in particular we let $r \in \N$ satisfy $(1-\frac{\xi}{2})^r < \xi$, which is a computable process. 

By construction of our F\o lner sequence, we can computably identify $n_0 \in \N$ such that for all $n \geq n_0$, $F_n$ is $(F, \lambda)$-invariant. We apply the computable process found in Lemma 4.3 of \cite{dhz} to get our starting collection of shapes $\{ F_{n_1}, \dots, F_{n_r} \}$. We note to the reader that in \cite{dhz}, these shapes induce a quasitiling of $G$ (one that exhausts $G$ but contains overlaps of the shapes). We may then remove the overlapping sections by cutting out at most a $(1-\lambda)$ proportion of them to form a $(1-\xi)$-tiling of $G$. For our purposes it is sufficient to examine only the shapes, of which there are only finitely many candidates. We now define the collection of possible shapes needed to $(1-\xi)$-tile $G$: 
$$\mathcal{S} = \bigcup_{i=1}^r \{ F^\circ \subset F_{n_i} : |F^\circ | \geq (1-\lambda) |F_{n_i}| \}. $$
First notice since we have a finite collection of finite sets, this is computable. Additionally, by choice of $\lambda$, since $(1-\lambda)^2 > 1-\xi$, there is a subset of shapes in $\mathcal{S}$ inducing a $(1-\xi)$-tiling of $G$. Finally, since $\lambda$ was taken sufficiently small with respect to Lemma \ref{lemma33} (Lemma 3.3 in \cite{dhz}), we know each $F^\circ \in \mathcal{S}$ is $(F, \xi)$-invariant. 

Since all steps involved in constructing $\mathcal{S}$ from $F$ and $\xi$ were computable in finite time, we can conclude the desired results. 
\end{proof}

We now apply the previous lemma to $F = E$ and $\xi = \eta / 2$ to arrive at our desired collection of shapes $\mathcal{S}$ inducing a $(1-\eta/2)$-tiling of $G$ and for which each shape is $(E, \eta/2)$-invariant.

\subsection{Approximating Pressure}\label{secapprox}

Note that by Theorem \ref{decidable language}, $X$ has decidable language, and since $\varphi$ is a locally constant, rational valued potential, for each $S_i \in \mathcal{S}$, $\frac{1}{|S_i|} \log Z_{S_i \backslash E^{-1} S_i^c}(\varphi)$ is computable. We may therefore compute the minimum. Fix $i \in \N$ such that $S_i$ attains this minimum, i.e.,
$$\frac{1}{|S_i|} \log Z_{S_i \backslash E^{-1} S_i^c}(\varphi) = \min \left\{\frac{1}{|S|} \log Z_{T \backslash E^{-1} S^c}(\varphi) : S \in \mathcal{S} \right\}. $$

\begin{lemma}\label{lem17} For our chosen $i$, we have: 
$$\frac{1}{|S_i \backslash S_i^cE^{-1}|} \log Z_{S_i \backslash S_i^cE^{-1}}(\varphi) \geq P_X(\varphi) \geq \frac{1-\eta}{|S_i|} \log Z_{S_i \backslash S_i^c E^{-1}}(\varphi). $$
\end{lemma} 

\begin{proof} Notice that since there exists a $(1-\eta/2)$-tiling of $G$ induced by $\mathcal{S}$, by definition for sufficiently large $n$, there exist finitely many pairwise disjoint translates of shapes from $\mathcal{S}$ whose union covers proportion $1 - \eta$ of $F_n$. We now examine $F_n$ for such $n$. 

For each $S_j \in \mathcal{S}$, let $n_j$ represent the number of translates of $S_j$ present in such an almost covering of $F_n$. Notice that since $E$ satisfies strong irreducibility for $X$, by definition of the partition function, and since $\varphi \geq 0$, we have 
$$Z_{F_n}(\varphi) \geq \prod_{S_j \in \mathcal{S}}  Z_{S_j \backslash S_j^c E^{-1}}(\varphi)^{n_j}. $$
By our choice of $i$ we have 
$$\frac{1}{|F_n|} \log Z_{F_n}(\varphi) \geq \frac{1-\eta}{|S_i|} \log  Z_{S_i \backslash S_i^c E^{-1}}(\varphi). $$
By taking limits as $n \rightarrow \infty$ combined with the infimum rule, we can conclude the desired result
$$\frac{1}{|S_i \backslash S_i^c E^{-1}|} \log Z_{S_i \backslash S_i^cE^{-1}}(\varphi) \geq P_X(\varphi) \geq \frac{1-\eta}{|S_i|} \log Z_{S_i \backslash S_i^cE^{-1} }(\varphi). $$
\end{proof}


We now note that since $S_i$ is $(1-\eta/2)$ invariant, $(1-\eta/2) < \frac{|S_i \backslash  S_i^c E^{-1} |}{|S_i|} \leq 1$. We now observe: 
$$\left| \frac{1}{|S_i \backslash S_i^cE^{-1} |} \log Z_{S_i \backslash S_i^cE^{-1} }(\phi) - \frac{1-\eta}{|S_i|} \log Z_{S_i \backslash S_i^cE^{-1} }(\phi) \right|  $$
$$= \left| \frac{1}{| S_i |} \log Z_{S_i \backslash  S_i^c E^{-1}}(\phi) \right| \cdot \left|   \frac{|S_i|}{|S_i \backslash S_i^c E^{-1}|}  - (1-\eta) \right| . $$
Since $|S_i| \geq |S_i \backslash  S_i^c E^{-1}|$ we know by the definition of the partition function 
$$\left| \frac{1}{| S_i |} \log Z_{S_i \backslash  S_i^c E^{-1}}(\phi) \right| \leq \left| \frac{1}{| S_i \backslash S_i^c E^{-1} |} \log Z_{S_i \backslash  S_i^c E^{-1}}(\phi) \right| \leq \log |\mathcal{A}| + ||\phi||_\infty. $$
Additionally since $(1-\eta/2) < \frac{|S_i \backslash S_i^c E^{-1}|}{|S_i|} \leq 1$, we know 
$$\left|   \frac{|S_i|}{|S_i \backslash S_i^c E^{-1}|}  - (1-\eta) \right|  \leq \left| \frac{1}{1-\eta/2} - (1-\eta) \right|. $$
We therefore arrive at the bound 
\[
\left| \frac{1}{|S_i \backslash S_i^c E^{-1}|} \log Z_{S_i \backslash  S_i^c E^{-1} }(\phi) - \frac{1-\eta}{|S_i|} \log Z_{S_i \backslash  S_i^c E^{-1}}(\phi) \right| \leq  \left( |\mathcal{A}| + ||\phi||_\infty \right) \cdot \left| \frac{1}{1-\eta/2} - (1-\eta) \right|.
\]

By our choice of $\eta$, we know this is bounded above by $\epsilon / 2$, and so by Lemma~\ref{lem17},

\begin{equation}\label{partition-estimate}
|P_X(\varphi) - \frac{1}{|S_i \backslash S_i^c E^{-1}|} \log Z_{S_i \backslash  S_i^c E^{-1}}(\varphi)| < \frac{\epsilon}{2}.
\end{equation}

We therefore have 
$$\left| P_X(\phi) - \frac{1}{|S_i \backslash S_i^c E^{-1} |} \log Z_{S_i \backslash  S_i^c E^{-1} }(\varphi) \right| $$
$$\leq \left| P_X(\phi) - P_X(\varphi) \right| + \left| P_X(\varphi) - \frac{1}{|S_i \backslash S_i^c E^{-1} |} \log Z_{S_i \backslash  S_i^c E^{-1}}(\varphi) \right| < \epsilon.$$

Since the approximation $\frac{1}{|S_i \backslash S_i^c E^{-1}|} \log Z_{S_i \backslash  S_i^c E^{-1}}(\varphi)$ is algorithmically computable as noted above, $P_X : C(\mathcal{A}^G) \rightarrow \R$ is computable, completing the proof of Theorem~\ref{pressure is computable}.

\begin{obs} In this proof, we assumed decidable word problem in order to identify shapes, as well as compute supremums over the locally constant potentials. If $G$ did not have decidable word problem, then the same conclusions should hold if one is assumed to have an oracle for the word problem. 
\end{obs}

\section{$P_X$ is Computable from Above Given a Forbidden List}

In this section we will show that even when our subshift $X$ is not SI or an SFT, pressure is still computable from above given an enumerated forbidden list inducing $X$. We again let $G$ be a finitely generated, amenable group with decidable word problem, and we fix a F\o lner sequence $\{ F_n \}$. Let $X \subset \mathcal{A}^G$ be any subshift and let $\mathcal{F} = \{ w_1, w_2, \dots \}$ be an enumeration of a forbidden list for $X$. We additionally denote $\mathcal{F}_n = \{ w_1 , \dots w_n \}$ to be the first $n$ forbidden words from our forbidden list. 

We remind the reader that for any finite $F \Subset G$ and $v \in \mathcal{A}^F$, we say $v$ is locally admissible (with respect to $\mathcal{F}$) if for every $w \in \mathcal{F}$, $v$ is not a subword of $w$. We define $LA_F(\mathcal{F})$ to be the set of all locally admissible $F$-words in $X$ according to $\mathcal{F}$. Notice here that $LA_F(\mathcal{F}) = \bigcap_{n \in \N} LA_F(\mathcal{F}_n )$ and since this collection is finite, for each $F \Subset G$ we have some $N \in \N$ such that $LA_F(\mathcal{F}) = LA_F(\mathcal{F}_N)$. 

The proofs below will rely on an adapted concept of $t$-extendable words used in \cite{rosenfeld-t-extend} by Rosenfeld. We first fix a computable enumeration of $G = \{ g_i : i \in \N \}$ such that $g_1 = e$. For each $t \geq 1$, define $G_t = \{ g_1, \dots, g_t \}$. 

\begin{defn} For all $F \Subset G$, $t, n \geq 1$, define $E_{t, n}^{(F)} \subset \mathcal{A}^F$ as follows. $w \in E_{t, n}^{(F)}$ if and only if the following conditions hold: 
\begin{itemize}
\item $w \in LA_F(\mathcal{F}_n)$, and 
\item there exists $v \in LA_{F G_t}(\mathcal{F}_n)$ such that $v|_F = w$. 
\end{itemize}
\end{defn}

\begin{lemma} For any $F \Subset G$, $L_F(X) = \bigcap_{t, n \in \N} E_{t, n}^{(F)} = E_{T, N}^{(F)}$ for some $T, N \in \N$. 
\end{lemma}

\begin{proof} Clearly for all $T, N \in \N$, $L_F(X) \subset \bigcap_{t, n \in \N} E_{t, n}^{(F)} \subset E_{T, N}^{(F)} $. Since each $E_{t, n}^{(F)}$ is finite, the intersection must stabilize and we therefore have $\bigcap_{t, n \in \N} E_{t, n}^{(F)} = E_{T, N}^{(F)} $ for some $T, N \in \N$. Finally, by a compactness argument it follows that $\bigcap_{t, n \in \N} E_{t, n}^{(F)} = L_F(X)$. 
\end{proof}

\begin{lemma}\label{modified partition lemma} Given a rational-valued $\phi \in C(\mathcal{A}^G)$ which is locally constant for $E \Subset G$ containing $e$,
$$Z_F(\phi) = \inf_{t, n \in \N} \hat{Z}^{(t, n)}_F(\phi)$$
where 
$$\hat{Z}^{(t, n)}_F(\phi) = \sum_{w \in E_{t, n}^{(F)}} \sup \left\{  exp \left( \sum_{g \in F} \phi ( \sigma_g ( v) ) \right) : v \in E_{t, n}^{(EF)} \text{ s.t. } v|_F = w \right\} . $$
\end{lemma}

\begin{proof} Define for each $F \Subset G$ the $(t, n)$-modified partition function $\hat{Z}^{(t, n)}_F(\phi)$ as in the statement of the lemma. Notice that since 
$L_F(X) \subset E_{t, n}^{(F)}$ for all $t, n \in \N$ and $\sum_{g \in F} \phi ( \sigma_g ( x) )$ depends only on $x|_{EF}$, we know $Z_F(\phi) \leq \hat{Z}^{(t, n)}_F(\phi)$. Additionally, since $L_F(X) = \bigcap_{t, n \in \N} E_{t, n}^{(F)} = E_{T, N}^{(F)}$ for some $T, N \in \N$, it immediately follows that $L_F(X) = \min_{t, n \in \N} \hat{Z}^{(t, n)}_F(\phi) = \inf_{t, n \in \N} \hat{Z}^{(t, n)}_F(\phi)$. 
\end{proof}

\begin{lemma}\label{locally constant comp from above} Given a rational-valued $\phi \in C(\mathcal{A}^G)$ which is locally constant for $E \Subset G$ containing $e$, $P_X(\phi)$ is computable from above.
\end{lemma}

\begin{proof} First notice that for each $t, n \in \N$, $F \Subset G$, and $v \in \mathcal{A}^F$, it is decidable in finite time if $v \in E_{t, n}^{(F)}$, and therefore $\hat{Z}^{(t, n)}_F(\phi)$ is computable. Further, since we know $P_X(\phi) = \inf_{n \in \N} |F_n|^{-1} \log Z_{F_n}(\phi)$, it follows immediately from Lemma \ref{modified partition lemma} that 
$$P_X(\phi) = \inf_{n, t, s \in \N} \frac{1}{|F_n|} \log \hat{Z}^{(t, s)}_{F_n}(\phi). $$
Since for all $n, t, s \in \N$, $\frac{1}{|F_n|} \log \hat{Z}^{(t, s)}_{F_n}(\phi)$ is computable, it follows that $P_X(\phi)$ is computable from above. 
\end{proof}

\begin{theorem}\label{pressure is computable from above} Let $G$ be a finitely generated, amenable group with decidable word problem and let $X \subset \mathcal{A}^G$ be any subshift. Then $P_X : C(\mathcal{A}^G) \rightarrow \R$ is computable from above from an enumeration of a forbidden list for $X$. 
\end{theorem}

\begin{proof} We let $\phi \in C(\mathcal{A}^G)$ and let $\lambda$ be an oracle for $\phi$. We now define a sequence of potentials: 
$$\psi_k =  \lambda(k) + \cdot 2^{-k}. $$

Notice that $\psi_k$ is a sequence of rational-valued potentials locally constant on shapes $E_k$ with both the shapes $E_k$ and values of $\psi_k$ are algorithmically computable from $\lambda$.  Additionally, $\psi_k \geq \phi$ for all $k$ and uniformly approach $\phi$, and so 
$$P_X(\phi) = \lim_{k \rightarrow \infty} P_X( \psi_k ) = \inf_{k \in \N} P_X( \psi_k ).$$
By application of Lemma \ref{locally constant comp from above} we notice that 
$$P_X(\phi) = \inf_{n, k, t, s \in \N} \frac{1}{|F_n|} \log \hat{Z}^{(t,s)}_{F_n}(\psi_k). $$
Since for all $n, k, t, s \in \N$, the quantity $\frac{1}{|F_n|} \log \hat{Z}^{(t,s)}_{F_n}(\psi_k )$ is computable, we can conclude that $P_X : C(\mathcal{A}^G) \rightarrow \R$ is computable from above. 
\end{proof}

\begin{remark} Notice that in the proof of Theorem \ref{pressure is computable from above}, the full strength of the oracle $\lambda$ for $\phi$ was not required. In fact, we only used the fact that the locally constant approximations $\lambda(k)$ are greater than $\phi$ and approach it uniformly (with no knowledge about rate of approximation/no ability to give lower bounds).

\end{remark}

\section{Computability Properties of Ground State Energy and Entropy}


In this section, we let $G$ be any countable, amenable group. In Section~\ref{energy} we will consider subshifts $X \subset \mathcal{A}^G$ for which $P_X : C(\mathcal{A}^G) \rightarrow \R$ is computable. For instance, by Theorem \ref{pressure is computable}, when $G$ is finitely generated with decidable word problem, any SI SFT is such an $X$. Or, in the case where $G = \Z$, $X$ can be taken to be any SFT, sofic subshift, or coded subshifts treated in \cite{pressure-comp-beyond-sft}, including many $S$-gap shifts, generalized gap shifts, and $\beta$-shifts.  We show that ground state energy is computable for any such $X$.

In Section~\ref{energy2}, we will assume only that $P_X : C(\mathcal{A}^G) \rightarrow \R$ is computable from above. By Theorem~\ref{pressure is computable from above}, this includes any subshift with an algorithmically enumerated forbidden list. In this setting, we show that ground state energy is computable from above. 

Finally, in Section \ref{groundentropy} we again consider subshifts $X \subset \mathcal{A}^G$ for which $P_X : C(\mathcal{A}^G) \rightarrow \R$ is computable. In this setting we show that the ground state entropy map is computable from above, extending a result from \cite{zero-temp-computable} which applied to $\mathbb{Z}$-SFTs.



\subsection{Computability of Ground State Energy}\label{energy}

We let $X \subset  \mathcal{A}^G$ be a subshift for which $P_X : C(\mathcal{A}^G) \rightarrow \R$ is computable. Additionally we fix  $\phi \in C(\mathcal{A}^G)$ and let $\gamma$ be an oracle for $\phi$. Of interest in this section is the map $P_{X,\phi} : (0, \infty) \rightarrow \R$ where $P_{X,\phi}(\beta) := P_X(\beta \phi)$, with $\beta$ representing the inverse temperature of our system. 

\begin{lemma}\label{Pphi computable} $P_{X,\phi}$ is a computable function given an oracle for $\phi$. 
\end{lemma}

\begin{proof} First let $T$ be a Turing machine constructing approximations to $P_X$ given an oracle for $\phi$, which exists by assumption. We now let $\beta > 0$ and let $\lambda$ be an oracle for $\beta$. First note that for all $n \in \N$,  
$$||\lambda(n) \gamma(n) - \beta \phi||_\infty \leq ||(\lambda(n) - \beta) \gamma(n) ||_\infty + \beta ||\gamma(n) - \phi||_\infty \leq 2^{-n} \left( ||\gamma(1)||_\infty + |\lambda(1)| + 2 \right) . $$
Since the quantity $||\gamma(1)||_\infty + |\lambda(1)| + 2$ is finite, we may remove finitely many terms from the oracles $\lambda$ and $\gamma$ to create $\lambda', \gamma'$ so that $\lambda' \cdot \gamma'$ is an oracle for $\beta \phi$. By definition of $T$, it is therefore the case that for any $k$,
$$| T(k, \lambda' \cdot \gamma') - P_{X,\phi}(\beta)| = | T(k, \lambda' \cdot \gamma') - P_X(\beta \phi)| < 2^{-k}. $$
We can therefore conclude that $P_{X,\phi}$ is computable. 
\end{proof}


%
%

\begin{theorem}\label{ground state energy} Let $G$ be a countable, amenable group and let  $X \subset \mathcal{A}^G$ be a subshift such that $P_X$ is a computable function. Then, the function sending $\phi$ to 
its ground state energy $\sup_{\nu \in M_\sigma(X)} \int \phi d\nu$ is computable. 
\end{theorem}

\begin{proof} 
We fix $\beta > 0$ and let $\nu_\beta$ be an equilibrium state for $\beta \phi$. Let $\mu$ be any ground state for $\phi$. By the variational principle combined with the fact that $\mu$ maximizes the integral of $\phi$, we know 
$$h(X) + \beta \int \phi d\mu \geq h(\nu_\beta) + \beta \int \phi d\nu_\beta \geq h(\mu) + \beta \int \phi d\mu . $$
Since $0 \leq h(\mu), h(\nu_\beta) \leq h(X)$ (again by the Variational Principle),
$$h(X) - h(\nu_\beta) \geq \beta \left( \int \phi d\nu_\beta - \int \phi d\mu \right) \geq h(\mu) - h(\nu_\beta). $$
Thus we have 
$$\left| \int \phi d\mu - \int \phi d\nu_\beta \right| \leq \frac{h(X)}{\beta}. $$

We therefore know 
$$\left| \frac{P_{X,\phi}(\beta)}{\beta} - \sup_{\nu \in M_\sigma(X)} \int \phi d\nu \right| = \left| \frac{P_{X,\phi}(\beta)}{\beta} - \int \phi d\mu \right| = \left| \frac{ h(\nu_\beta)}{\beta} + \int \phi d\nu_\beta - \int \phi d\mu \right| \leq \frac{2 h(X)}{\beta}. $$

It therefore follows that for every $\epsilon > 0$, if we take
$\beta \geq \lceil \frac{4 \log |\mathcal{A}|}{\epsilon} \rceil \geq \frac{4 h(X)}{\epsilon}$, then any $\epsilon/2$-approximation to $\beta^{-1} P_\phi(\beta)$ is within $\epsilon$ of $\sup_{\nu \in M_\sigma(X)} \int \phi d\nu$. By Lemma~\ref{Pphi computable}, such approximations are algorithmically computable, completing the proof. 
\end{proof}

\subsection{Computability from Above of Ground State Energy}\label{energy2}

In this section, we will show that for any subshift $X$ for which $P_X : C(\mathcal{A}^G) \rightarrow \R$ is computable from above, the map sending $\phi$ to $\sup_{\nu \in M_\sigma(X)} \int \phi d\mu$ is computable from above. We begin with the following lemma: 

\begin{lemma} For all $\phi \in C(\mathcal{A}^G)$, $\sup_{\nu \in M_\sigma(X)} \int \phi d\mu = \inf_{\beta > 0} \beta^{-1} P_{X,\phi}(\beta)$. 
\end{lemma}

\begin{proof}  We now let $\alpha > \beta > 0$ and let $\nu_\alpha, \nu_\beta$ be equilibrium states for $\alpha \phi$ and $\beta \phi$ respectively. We now examine 
$$\beta^{-1} P_{X,\phi}(\beta) = \beta^{-1} \left( h(\nu_\beta) + \int \beta \phi d\nu_\beta \right) \geq \beta^{-1} \left( h(\nu_\alpha) + \int \beta \phi d\nu_\alpha \right) $$
$$= \beta^{-1} h(\nu_\alpha) + \int \phi d\nu_\alpha \geq \alpha^{-1} h(\nu_\alpha) + \int \phi d\nu_\alpha = \alpha^{-1} P_{X,\phi}(\alpha). $$
We can therefore conclude the function $ \beta \mapsto \beta^{-1} P_{X,\phi}(\beta)$ is non-increasing for $\beta > 0$, and thus 
$$\sup_{\nu \in M_\sigma(X)} \int \phi d\mu = \lim_{\beta \rightarrow \infty} \beta^{-1} P_{X,\phi}(\beta)  = \inf_{\beta > 0} \beta^{-1} P_{X,\phi}(\beta). $$
\end{proof}

\begin{corollary}\label{computable from above cor ground state}  Let $G$ be a countable, amenable group and let $X \subset \mathcal{A}^G$ be any subshift such that $P_X : C(\mathcal{A}^G) \rightarrow \R$ is computable from above. Then the function sending $\phi$ to its ground state energy $\sup_{\nu \in M_\sigma(X)} \int \phi d\mu$ is computable from above.  
\end{corollary}

\begin{proof} Similar to Lemma~\ref{Pphi computable}, since $P_X$ is computable by above it follows that $P_{X,\phi}(\beta )$ is computable from above (and therefore $\beta^{-1} P_{X,\phi}(\beta)$ is also computable from above). Since $\sup_{\nu \in M_\sigma(X)} \int \phi d\mu = \inf_{\beta > 0} G_\phi(\beta)$, and $G_\phi(\beta)$ is continuous and non-increasing, we also know $\sup_{\nu \in M_\sigma(X)} \int \phi d\mu = \inf_{n \in \N} n^{-1} P_{X,\phi}(n)$. It immediately follows that $\sup_{\nu \in M_\sigma(X)} \int \phi d\mu$ is computable from above. 
\end{proof}

\begin{corollary}\label{computable from above cor2 ground state}  In particular, this implies that if $G$ is finitely generated with decidable word problem and $X$ has a computable forbidden list, then the function sending $\phi$ to its ground state energy $\sup_{\nu \in M_\sigma(X)} \int \phi d\mu$ is computable from above.  Otherwise, by Theorem \ref{pressure is computable from above}, If $\lambda$ is an oracle for the forbidden list of $X$, then the ground state energy function is computable from above given $\lambda$. 
\end{corollary}

\subsection{Computability from Above of Ground State Entropy}\label{groundentropy}

We now remind the reader that we call the ground state entropy of a potential $\phi$ the common value of the entropy for ground states for $\phi$. In particular, this is the maximal entropy among invariant measures which maximize the integral of $\phi$.  

\begin{corollary}\label{ground state entropy} Let $G$ be a countable, amenable group and let  $X \subset \mathcal{A}^G$ be a subshift such that $P_X: C(\mathcal{A}^G) \rightarrow \R$ is  computable. Then, the function that sends $\phi$ to the ground state entropy of $\phi$ is computable from above.  
\end{corollary}

\begin{proof} First let $\mu$ be any ground state for $\phi$. First notice that $P_{X,\phi}(\beta)$ approaches from above an asymptote function $h(\mu) + \beta \sup_{\nu \in M_\sigma(X)} \int \phi d\nu$.  In particular,  
$$P_{X,\phi}(\beta) \geq h(\mu) + \beta  \int \phi d\mu .  $$
Our upper approximation of $h(\mu)$ can be identified by 
$$P_{X,\phi}(\beta) - \beta  \int \phi d\mu, $$
which converges from above to $h(\mu)$ as $\beta \rightarrow \infty$. Since $P_{X,\phi}(n)$ and $n \int \phi d\mu$ are computable for all $n \in \N$ (by assumption and by Theorem~\ref{ground state energy}, respectively), we may compute $h(\mu)$ from above.
\end{proof}

\begin{corollary}\label{si sft ground state entropy} For an SI SFT $X$ over a finitely generated, amenable group $G$ with decidable word problem, the function that sends $\phi$ to the ground state entropy of $\phi$ is computable from above.  
\end{corollary}

Unfortunately, as noted by Kucherenko and Quas, we have no control on the rate of the convergence to $h(\mu)$. In fact, in the 
$G = \Z$ setting they showed in \cite{kucherenko-quas} that $P_{X,\phi}(\beta) - \beta \int \phi d\mu$ can converge to $h(\mu)$ at most exponentially quickly for H\o lder continuous potentials that are not cohomologous to a constant. Further, the rate of convergence can be arbitrarily slow. In fact, the map is not computable in general:

\begin{obs} There exists a full-shift $\mathcal{A}^{\Z^2}$ for which the map that sends $\phi$ to ground state entropy is not computable. 
\end{obs}

To see this, take any alphabet $\mathcal{A}$ for which $\mathcal{A}^{\Z^2}$ contains SFTs with noncomputable entropies, as found in \cite{hoch-meyer}. We will provide an example of a locally constant, integer valued potential $\phi$ on the full $G$-shift for which the ground state entropy is not computable. 

Let $X \subset \mathcal{A}^{\Z^2}$ be any SFT such that $h(X)$ is not computable and assume without loss of generality that it is induced by a forbidden list $\mathcal{F} \subset A^F$ for some $F \Subset G$. Define a potential which is locally constant for $F$ by $\phi(x) = -1$ if $x|_F \in \mathcal{F}$ and $\phi(x) = 0$ otherwise.

Suppose that $\mu$ is a ground state for $\phi$. Then $\mu$ must maximize the integral of $\phi$, meaning that it must have $\mu(X) = 1$ (any such measure satisfies $\phi = 0$ $\mu$-a.e., achieving the maximum integral of $0$). The ground state entropy is the maximal entropy among such measures, which is $h(X)$ by the Variational Principle.

Since $\phi$ was locally constant and integer valued, ground state entropy cannot possibly be a computable function in general.






\section{Computation Time Bounds for a $\Z^d$ SI SFT}


\begin{theorem}\label{compute time}
For any $d$, there exists a Turing machine $T$ that, upon input of: 
\begin{itemize} 
\item an SI SFT $X \subset \mathcal{A}^{\Z^d}$, 
\item an oracle for the language of $X$, and
\item a rational-valued single-site potential $\phi$,
\end{itemize} computes $P_X(\phi)$ to precision $2^{-k}$ in $|\mathcal{A}|^{O(2^{kd})}$ time. 
\end{theorem}

Alternately phrased, computing to within $\epsilon$ takes $e^{O(\epsilon^{-d})}$ time. 
We note that although we have phrased this result for single-site $\phi$ for simplicity, the same holds for any locally constant potential by using a higher-block recoding of $X$; see the proof of Theorem~\ref{compute time2} for an example of this technique.

\begin{proof} Let $X$, $\phi$, and $n$ be as in the theorem. First note that we may compute $||\phi||_\infty$ by trying all possible symbols in the language, so this has computation time on the order of $ |\mathcal{A}|$. We now let 
$$\eta = \frac{2^{-k-1}}{3 \log | \mathcal{A}| + 3 ||\phi||_\infty + 2^{-k} }$$
and note it can be shown that this $\eta > 0$ satisfies 
$$\frac{1}{1-\eta/2} - (1-\eta) < \frac{2^{-k}}{\log |\mathcal{A}| + ||\phi||_\infty}. $$
We now let $M = \lceil \frac{dn}{\eta} \rceil - n$, and notice that $[-(M+n), (M+n)]^d$ is $([-n, n]^d, \eta)$-invariant. We note that  $E = [-n, n]^d$ and $T_i = [-(M+n), (M+n)]^d$ satisfy the hypotheses of Section~\ref{secapprox}, and so by (\ref{partition-estimate}) we can approximate $P_X(\phi)$ to $2^{-k}$ precision by computing the following: 
\begin{equation}\label{approx}
\frac{1}{(2M+1)^d} \log Z_{[-M,M]^d}(\phi) =
\frac{1}{(2M+1)^d} \log \left( \sum_{v \in L_{[-M, M]^d}(X) } exp \left( \sum_{g \in [-M, M]^d} \phi(\sigma^g(v)) \right) \right).
\end{equation}
(We note that we do not need the $\sup$ from the definition of $Z_{[-M,M]^d}(\phi)$ here since $\phi(\sigma^g(v))$ depends only on $v(g)$, and so is independent of $x \in [v]$.
It is therefore sufficient to compute $\sum_{g \in [-M, M]^d} \phi(\sigma^g(v))$ for all $v \in L_{[-M, M]^d}(X)$. Since there are at most  $\left| \mathcal{A} \right|^{(2M+1)^d}$ of such words, this bounds the number of times we must call $\phi$ by 
$$\left| \mathcal{A} \right|^{(2M+1)^d}(2M+1)^d. $$
We will then need to compute the $exp$ function at most $\left| \mathcal{A} \right|^{(2M+1)^d}$ times. Finally we take log of the resulting sum and we divide by $(2M+1)^d$ and the resulting number is our pressure approximation. 

The above described algorithm above has computation time on the order of 
$$\left| \mathcal{A} \right|^{(2\frac{dn}{\eta} -2n +1)^d}\left(2\frac{dn}{\eta} -2n +1\right)^d. $$

Notice here by our choice of $\eta$ we have 
$$2 \frac{dn}{\eta}  =2^{k+2}  3dn(\log | \mathcal{A}| +  ||\phi||_\infty) + 4. $$

Thus we have the computation time is approximately 
$$\left| \mathcal{A} \right|^{(2^{k+2}  3dn(\log | \mathcal{A}| +  ||\phi||_\infty) -2n +5)^d}(2^{k+2}  3dn(\log | \mathcal{A}| +  ||\phi||_\infty) -2n + 5)^d. $$
Thus, computing a fixed locally constant potential up to tolerance $2^{-k}$ takes on the order of $|\mathcal{A}|^{O(2^{kd})}$ computation time. 
\end{proof}

Finally, we prove that the special case $X = A^{\mathbb{Z}^2}$ can be done in significantly reduced computation time (i.e. approximation to within $\epsilon$ in singly exponential time in $\epsilon^{-1}$). 

\begin{theorem}\label{compute time2}
There exists a Turing machine $T$ that, upon input of a rational-valued locally constant potential $\phi$ on a full shift $\mathcal{A}^{\mathbb{Z}^2}$, computes $P_X(\phi)$ to precision $2^{-k}$ in $|\mathcal{A}|^{O(2^k)}$ time. 
\end{theorem}

\begin{proof}
First, we apply a so-called higher block recoding so that we can assume without loss of generality that $\phi$ is single-site. More specifically, if $\phi(x)$ depends only on $x([-n,n]^2)$, then we define a new $\mathbb{Z}^2$ nearest-neighbor SFT $X_n$ as follows: the alphabet $A_n := A^{[-n,n]^2}$, and for $a,b \in A_n$, $ab$ is a legal horizontal transition iff the rightmost $2n$ columns of $a$ are equal to the leftmost $2n$ columns of $b$, and $\begin{smallmatrix} a\\ b \end{smallmatrix}$ is a legal vertical transition iff the bottom $2n$ rows of $a$ are equal to the top $2n$ rows of $b$. It's easily checked that $X_n$ is always SI, and that $P_{A^{\mathbb{Z}^2}}(\phi) = P_{X_n}(\widetilde{\phi})$, where $\widetilde{\phi}$ is the single-site potential defined by $\widetilde{\phi}(a)$ equal to the common value of $\phi$ on all $x \in [a]$.

It now suffices to show that for $X_n$, (\ref{approx}) from the proof of Theorem~\ref{compute time} can be computed in time $|\mathcal{A}|^{O(2^k)}$. We do this via a `transfer matrix' method used in several works, the presentation here is closest to that used in \cite{marcus-pavlov}. 

Define, for any $M$, a matrix $B(M)$ with entries indexed by words in $L_{X_n}({\{0\} \times [-M,M]})$ ((2M+1)-high columns). 
For $a = \begin{smallmatrix} a_M\\ \cdots\\ a_{-M} \end{smallmatrix}, b = \begin{smallmatrix} b_M\\ \cdots\\ b_{-M} \end{smallmatrix} \in A_n^{\{0\} \times [-M,M]}$, the entry $B(M)_{ab}$ is equal to $0$ if $ab = \begin{smallmatrix} a_M b_M\\ \cdots\\ a_{-M} b_{-M}\end{smallmatrix} \notin L(X_n)$, and is equal to $\prod_{i=-M}^M \textrm{exp } \widetilde{\phi}(a_i)$ otherwise.

Then it's easily checked by induction that the sum of all entries of $B(M)^{2M+1}$ is equal to the sum over all $v \in L_{[-M,M+1] \times [-M,M]}(X_n)$ of 
\[
\prod_{-M \leq i,j \leq M} \textrm{exp } \widetilde{\phi}(v(i,j)) = 
\textrm{exp } \left(\sum_{g \in [-M,M]^2} 
\widetilde{\phi}(\sigma^g v) \right).
\]
This very nearly yields (\ref{approx}) (upon taking a logarithm and dividing by $(2M+1)^2$), except for the issue that we are summing over $v$ with shape $[-M,M+1] \times [-M,M]$ rather than $[-M,M]^2$. This is easily dealt with though; patterns in $L_{[-M,M]^2}(X_n)$ are in clear bijective correspondence to patterns in $\mathcal{A}^{[-M-n,M+n]^2}$, and so every pattern in $L_{[-M,M]^2}(X_n)$ has exactly $|A|^{2M+2n+1}$ extensions to the right to create legal patterns in
$L_{[-M,M+1] \times [-M,M]}(X_n)$. Therefore, (\ref{approx}) is obtained by dividing the sum of the entries of $B(M)^{2M+1}$ by $|A|^{2M+2n+1}$, taking a logarithm, and dividing by $(2M+1)^2$.

It remains only to recall that $M = O(2^k)$ and note that the size of $B(M)$ is $|A_n| = |A|^{(2n+1)(2M+2n+1)} = |A|^{O(2^k)}$, and that matrix multiplication is polynomial time, so computing the $(2M+1)$th power takes $|A|^{O(2^k)}$ time as well.

\end{proof}

We note that a similar technique can be used for $d > 2$, where the matrix is indexed by $(d-1)$-dimensional configurations and so the computation time is $|A|^{O(2^{(d-1)k})}$, meaning that the algorithm computes to within $\epsilon$ in $e^{O(\epsilon^{-(d-1)})}$ time. 
The main obstacle to applying this to SI SFTs other than the full shift is that we cannot know whether the $v$ summed over when one sums entries of $B(M)^{2M+1}$ are globally admissible, only that each pair of columns form legal adjacencies. We imagine that a similar technique could be used if all locally admissible square patterns in $X$ are globally admissible, but leave details to the interested reader.




\bibliographystyle{plain}
\bibliography{mybib}

\end{document}